\documentclass[headsepline, auto-pst-pdf, a4paper ,openany, 11pt]{amsart}
\usepackage{a4wide}
\usepackage{amsthm,amssymb,amsfonts}
\usepackage{mathtools}
\usepackage{epsfig}
\usepackage{graphicx,color}
\usepackage{verbatim,enumerate}
\usepackage[T1]{fontenc}
\usepackage[ansinew]{inputenc}
\newtheorem{theorem}{Theorem}[section]
\newtheorem{thm}[theorem]{Theorem}
\newtheorem{lem}[theorem]{Lemma}
\newtheorem{cor}[theorem]{Corollary}
\newtheorem{obs}[theorem]{Observation}

\newtheorem{prop}[theorem]{Proposition}

\theoremstyle{remark}
\newtheorem{remark}[theorem]{Remark}

\newcommand{\eps}{\varepsilon}
\newcommand{\prob}{\ensuremath{\mathbb{P}}}
\newcommand{\expec}{\ensuremath{\mathbb{E}}}

\newcommand{\var}{\ensuremath{\mathbb{V}}}

\newcommand{\real}{\ensuremath{\mathbb{R}}}

\newcommand{\Gnp}{\ensuremath{\mathcal{G}(n,p)}}
\newcommand{\Ln}{\ensuremath{\log_q}}
\newcommand{\myr}{\ensuremath{\hat{r}}}
\newcommand{\rt}{\right}
\newcommand{\lt}{\left}
\newcommand{\Ocal}{\ensuremath{\mathcal{O}}}
\newcommand{\Bin}{\ensuremath{\mathrm{Bin}}}

\definecolor{nicegreen}{RGB}{0,204,0}
\definecolor{myred}{RGB}{220,24,10}

\begin{document}
\title{On the Concentration of the Domination Number of the Random Graph}
\author[R.~Glebov]{Roman Glebov}
\address{Department of Mathematics, ETH, 8092 Zurich, Switzerland}
\email{roman.glebov@math.ethz.ch}
\thanks{The first author was supported by DFG within the research training group ``Methods for Discrete Structures".}
\author[A.~Liebenau]{Anita Liebenau}
\address{Department of Computer Science, University of Warwick, UK.}
\email{a.liebenau@warwick.ac.uk}
\thanks{The second author was supported by DFG within the graduate school Berlin Mathematical School.}
\author[T.~Szab\'o]{Tibor Szab\'o}
\address{Institut f\"{u}r Mathematik, Freie Universit\"at Berlin, Arnimallee 3-5, D-14195 Berlin, Germany}
\email{szabo@math.fu-berlin.de}
\thanks{Research of the third author was 
partially supported by DFG within the research training group 
``Methods for Discrete Structures''.}

\maketitle

\date{\today}
\begin{abstract}
In this paper we study the behaviour of the do\-mi\-na\-tion number of the Erd\H os-R\'enyi 
random graph $\Gnp$. Extending a result of
Wieland and Godbole we show that the domination number of $\Gnp$ is equal to one of two values 
asymptotically almost surely whenever $p \gg \frac{\ln^2n}{\sqrt{n}}$. The explicit values are exactly
at the first moment threshold, that is where the expected number of dominating sets starts to tend to infinity.
For small $p$ we also provide various non-concentration results which indicate why some sort of lower bound 
on the probability $p$ is necessary in our first theorem. 
Concentration, though not on a constant length interval, is proven for every $p\gg 1/n$.
These results show that unlike in the case of $p \gg \frac{\ln^2n}{\sqrt{n}}$  where
 concentration of the domination number happens around the first moment threshold,  for $p = O( \ln n/n)$ 
it does so around the median. In particular, in this range the two are far apart from each other.
\end{abstract}
 
\section{Introduction}

As usual, $\Gnp$ denotes the homogeneous Erd\H os-R\'enyi random graph 
model with $n$ labeled vertices in which edges are inserted independently with probability $p=p(n)$. A statement about $G \sim \Gnp$ is said to hold
{\em asymptotically almost surely (a.a.s.)} if
it holds with probability tending to 1 as $n \rightarrow \infty$.

An interesting phenomenon in random graph theory is that  
a.a.s.~many natural graph parameters tend to take their values on $\Gnp$ in 
a relatively short interval within their potential range, often around their
expectation.
Throughout the last three decades several of these parameters 
were shown to exhibit a very strong 
concentration, taking one of only two values  a.a.s.
The first such result is due to Matula \cite{m1972} who proved that 
when $p$ is constant then with probability tending to $1$ a graph $G\sim\Gnp$ 
has independence number $k-1$ or $k$, where $k=k(n,p)$ is an integer given by 
some explicit formula. The value of $k$ is determined by the simple first moment 
upper bound: it is the largest integer where the expected number of independent sets of that 
size does not anymore tend to $0$.
An extension of this to edge probabilities $p=\Omega (1/n)$, 
showing concentration of the independence number around this first moment bound, 
though not on a constant length interval, was obtained by Frieze~\cite{frieze1990}. 
Shamir and Spencer \cite{ss1987} were the first to prove a 
concentration result on the chromatic number of $\Gnp$. They showed that for arbitrary 
$p$ the chromatic number is contained in an interval of length $O(\sqrt{n})$ 
while for $p\leq n^{-\alpha}$ with $\alpha>5/6$ it is concentrated on an interval of 
just five integers a.a.s. This was strengthened to a two-point concentration 
by \L uzcak \cite{l1991} for any fixed $\alpha>5/6$, and later by Alon and Krivelevich \cite{ak1997}  for any fixed $\alpha > 1/2$.  
 The asymptotic formula for the value of the chromatic number 
was established by Bollob\'as~\cite{Bollobas1988} for constant $p$ and 
later by \L uczak~\cite{Luczak1991}  for all $p \gg 1/n$ to be equal to 
the simple lower bound given by the simple first moment upper bound on the 
independence number. 
The exact value of the chromatic number is still widely open for most values of $p$. 
For $p=d/n$, where $d>0$ is a fixed constant, Achlioptas and Naor \cite{an2005} 
pinned down the value $\chi(\Gnp)$ to be one of two precisely defined integers, while
for roughly ``half of the possible $d$'' they determined the unique value that $\chi (\Gnp)$
takes a.a.s. 
Coja-Oghlan and Vilenchik~\cite{CV} managed to extend this to the set of essentially all
$d$ (a set of relative density $1$). 
For $p <  n^{-\alpha}$, with $\alpha > 3/4$, Coja-Oghlan, Panagiotou, Steger~\cite{CPS}
determined the precise value of $\chi(\Gnp)$ up to three integers.
In a different model, M\"uller \cite{m2008} showed that for certain parameter values 
the clique number, the chromatic number, the degeneracy and the maximum degree 
of the random geometric graph are concentrated on two consecutive integers a.a.s.

In this paper we study the behaviour of the do\-mi\-na\-tion number of 
$\Gnp$ for various values of $p$, focusing in particular on the range of $p$ where
concentration on an interval of constant length might  happen. 
In  a graph $G=(V,E)$, we call a set $S\subseteq V$ {\em dominating} if every
vertex $v \in V$ is either a member of $S$, or adjacent to a member of $S$.
The {\em domination number D(G)} is the smallest cardinality of a dominating set in $G$.
Dominating sets and the domination number were
well-studied concepts of graph theory \cite{HHS} even before their importance in
theoretical computer science became apparent. 
Deciding the domination number being less than $k$ 
is one of the classic NP-complete problems~\cite{Garey-Johnson},
while dominating sets and its variants (e.g. connected dominating sets) are 
fundamental, e.g., in distributed computing, routing, and networks.

Early results on the concentration of the domination number of $G \sim \Gnp$
include the case when $p$ is fixed (see, for example,~\cite{ns1994} and~\cite{d2000}),
or when $p$ tends to 0 sufficiently slowly.
In this direction, Wieland and Godbole~\cite{wg2001} showed that  under the condition that
either $p$ is constant, or $p$ tends to 0 with
\[	 p = p(n) \geq 10 \sqrt{\frac{\ln \ln n}{\ln n}} , \]
the domination number $D(G)$ takes one of two consecutive integer values with probability tending to $1$,
as $n$ tends to infinity.
In \cite{wg2001} it is raised as an open problem
whether the validity of this two-point concentration result can be extended to a wider range of $p$.
In our main theorem we extend this range down to $p \gg \frac{\ln^2 n}{\sqrt{n}}$ and also include the range when
$p\rightarrow 1$.

Here and in the rest of the paper, we denote $q = \frac{1}{1-p}$ and $d=np$.

\begin{thm}
\label{MAIN}
Let $p=p(n)$ be such that $\frac{\ln^2n}{\sqrt{n}}\ll p < 1$, and let $G \sim \Gnp$.
Then there exists an $\myr = \myr (n,p(n))$, which is of the form
\[ \myr (n,p(n)) = \Ln \left(\frac{n \ln q}{\ln ^2 d} (1+o(1)) \right), \]
such that $D(G) = \lfloor \myr \rfloor + 1$ or $D(G) = \lfloor \myr \rfloor + 2$ a.a.s.
\end{thm}
\noindent
Note that, when $p \rightarrow 0$ and hence $\ln q = p(1+o(1))$ then the $\myr$ from Theorem~\ref{MAIN} is of the order
$\frac{\ln d}{p} = \frac{n \ln d}{d}$.

The choice of the value of $\myr$ in Theorem~\ref{MAIN} as the
start of the concentration interval will be a very natural one:
$\myr$ will represent a particular critical dominating set size, such that
the expected number of dominating sets of size $\lfloor \myr \rfloor$
tends to $0$, whereas the expected number of dominating sets of size
$\lfloor \myr \rfloor +2$ tends to $\infty$ very fast. Observe that this 
puts Theorem~\ref{MAIN} in line with earlier strong concentration results,
mentioned above about the independence number and the chromatic number,
where the critical value is also around the first moment threshold.

For technical reasons the value of $\myr$ will be defined somewhat implicitly
in the range when $p\rightarrow 0$ and $d\rightarrow \infty$:
\begin{align} \label{rDefinition}
	\myr = \min \lt\{\, r\ | \ \expec(X_{r}) \geq \frac{1}{d} \rt\} -1 .
\end{align}
For $p \rightarrow 1$ (and $p< 1$), we will show that the theorem holds with the explicit formula
$\myr = \Ln \left( \frac{n \ln q}{\ln^2 n}\right) $.

Similarly to  \cite{wg2001}  we use standard first and second moment methods to prove
the two-point-concentration result of Theorem~\ref{MAIN}, however the
technical difficulties increase significantly, hence much of the improvement goes into 
the fine asymptotics.
A lower bound of $n^{-1/2} {\rm polylog}\, n$ on $p$ seems to be the boundary of
these calculations.

One can nevertheless show some, though not constant-length,
concentration of $D(\Gnp)$ for smaller $p$ as well.
\begin{prop}
\label{SomeConcentration}
For $p\rightarrow 0$ and $d\rightarrow \infty$ we have that
\[\prob \left( D\left( \Gnp \right) = n\frac{\ln d}{d} (1+o(1))\right) \rightarrow 1.\]
\end{prop}

If $p$ tends to 0 faster than $n^{-1/2}\ln n$, then
one can decrease the length of the concentration interval using Talagrand's Inequality.
\begin{thm}
\label{MAINTAL}
Let $m = m(n)$ be the median of $D(G)$ where $G\sim \Gnp$ and let $t= t(n) \gg \sqrt{n}$.
Then $\prob (|D(G)-m| > t) \rightarrow 0$.
\end{thm}
\noindent
By Proposition~\ref{SomeConcentration} $m \approx n\frac{\ln d}{d}$
so Theorem~\ref{MAINTAL} constitutes an improved concentration result whenever 
$p \ll \frac{\ln n}{\sqrt{n}}$.
However, for larger $p$, the length $t \gg \sqrt{n}$  of the interval of concentration
is of larger order than the median.

In Theorem~\ref{MAIN} we show that for $p\gg\ln^2n/\sqrt{n}$ the domination number 
of {\Gnp} is concentrated
on two integers $\lfloor \myr \rfloor +1$ and $\lfloor \myr \rfloor +2$ a.a.s., where
$\myr$ is the size when the expected number of dominating sets changes
from tending to zero to tending to infinity.
It is natural to ask whether the lower bound, or at least some
sort of lower bound, on $p$ in Theorem~\ref{MAIN} is justified.
It is not hard to see that the theorem cannot be extended to hold for arbitrary $p$.
For $p\ll n^{-4/3}$, for example, we have that {\Gnp} consists of a
collection of vertex-disjoint stars (with at most two edges) a.a.s.,
hence its domination number equals $n$ minus its number of edges. 
Consequently $D({\Gnp})$ is not concentrated on any interval of length 
$ o\left( \sqrt{{n\choose 2}p}\right) $ a.a.s.
Our last theorem extends this to every $p=o(1/n)$, and provides various other
non-concentration results for larger $p$.
In particular we show that in a certain range of $p$
the domination number of \Gnp\ is not concentrated around the critical first moment threshold $\myr$.

\begin{thm}\label{newNonConc}
Let $G\sim \Gnp$. 
Let $\myr (n, p(n)) =\myr = \min\{ r : \expec(X_{r}) \geq 1\} -1 $.

\begin{itemize}
\item[$(a)$] For every $c,K >0$, if $p\leq \frac{c}{n}$ and $I\subseteq [n]$, 
$|I| < Kn\sqrt{p}$, then $\prob(D(G)\in I)\not\rightarrow 1$.
\item[$(b)$] For all $c>0$ there is $\eps >0$ such that for $\frac{c}{n}\leq p \ll 1$ 
	it holds that $D(G) > \myr + \eps n\exp[-2np]$ a.a.s.  (where $\eps$ can be chosen $1$ for any $c>1$.)
\item[$(c)$] 
For every $c>0$ and for every $p=p(n)$ with
$ \frac{1}{n}\ll p \ll 1$, $\prob(D(G) \leq \myr+c\frac{\myr}{n\sqrt{p}}) \not\rightarrow 1$.
\end{itemize}
\end{thm}
We want to remark that the value of $\myr$ in the theorem might differ by one from the definition in (\ref{rDefinition}) for the range $d\rightarrow \infty$ (this will be apparent from Lemma~\ref{expecJump}),
but this is not a concern here due to the asymptotic nature of the statement. 

Observe that part (a) implies in particular that the length of the concentration interval in Theorem~\ref{MAINTAL} is best possible in general. 
If $p=1/n$ for example, then the domination number $D(\Gnp)$ is not concentrated on any interval of length $O(\sqrt{n})$ a.a.s., but according
to Theorem~\ref{MAINTAL} it is concentrated on some interval of length
$\sqrt{n}f(n)$ for any function $f(n)$ tending to infinity.  
An analogous  non-concentration statement about the chromatic number 
is mentioned in the concluding remarks of \cite{ak1997}: for $p$ as high as, 
say, at least  $1- 1/10n$, there is no interval of length smaller than $\Omega(\sqrt{n})$ 
containing the chromatic number a.a.s.

Note that part (b) implies in particular that for $\frac{1}{n}\leq p \leq \frac{\ln n}{5 n}$ 
the probability that the domination number falls in some interval 
of length $ne^{-2np} \geq n^{0.6}$ around the first moment threshold $\myr$ tends to $0$.
Now since by Theorem \ref{MAINTAL} the domination number  $D(\Gnp)$ {\em is} concentrated a.a.s.~on an interval of length $n^{0.51}$ around the median $m$, this  interval of concentration has to be 
somewhat far apart from $\myr$ and in particular the distance between the median $m$ and the first moment threshold 
is at least $ne^{-2np}- n^{0.51} = \Omega (ne^{-2np})$.

For the range $p \gg \frac{\ln n}{n}$ the statement of part (b) becomes trivial, but 
part (c) is meaningful also to probabilities in this range and extends for example 
the non-concentration on any constant-length interval around the first moment threshold. 
That is, for probabilities $p \ll (\ln n / n)^{2/3}$, the domination number $D(\Gnp)$
is not concentrated  a.a.s.~on any constant length interval around $\myr$, since
$\frac{\myr}{n\sqrt{p}} \sim \frac{\ln d}{np^{3/2}} \gg 1$.

\subsection*{Notation and structure of the paper.} \ \\
In Section~\ref{SEC:exp}, we examine
the expected number of dominating sets of a particular size $r$ in $G\sim \Gnp$.
In Section~\ref{sec:2ptConc} we prove Theorem~\ref{MAIN}.
We split the proof according to whether $p \rightarrow 0$ or $p \rightarrow 1$.
Section~\ref{SEC:furtherStuff} contains the proofs of Proposition~\ref{SomeConcentration} and Theorem~\ref{MAINTAL}.
In Section~\ref{SEC:nonConc} we prove Theorem~\ref{newNonConc}.
Finally, in Section~\ref{conclusion} we discuss some possible extensions of these results.

%
%
\section{Expectation}
\label{SEC:exp}
In this section we study the expected number of dominating sets of size $r$ in the random graph
$\Gnp$ when $p \rightarrow 0$ and when $p \rightarrow 1$.
In particular, we are interested in the value of $r$ when the expectation
first exceeds $1$ and how fast it grows around this point.
Let $G \sim \Gnp$.
We denote by $X_r$ the
number of dominating sets of size $r$ in $G$.
For any fixed $r$-subset $S$ of $[n]$ and vertex $v\in [n]\setminus S$
the probability that $v$ is not dominated by $S$ in $G$ is $(1-p)^{r}$.
These events are mutually independent for a fixed $S$, hence
the probability that $S$ is dominating is $\left(1-(1-p)^r\right)^{n-r}$
and for the expectation of $X_r$ we have
 \[ \expec (X_r)={n\choose r}\left(1-(1-p)^r\right)^{n-r}.\]
It turns out that $\expec (X_r)$ first exceeds 1 when $r$ is in the range of
$\Ln \left(\frac{n\ln q}{\ln^2 d}\, (1+o(1))\right)$.
Recall that we use the notation $d = np$ and $q=\frac{1}{1-p}$.
We now split the analysis according to whether $p \rightarrow 0$ or $p \rightarrow 1$.
\subsection{The sparse case}
Note that when $p \rightarrow 0$, $\ln q=p(1+o(1))$.
We thus have the identity
$\Ln \left(\frac{n\ln q}{\ln^2 d}\, (1+o(1))\right) = \Ln \left(\frac{d}{\ln^2 d} (1+o(1)) \right)$.
The following two small calculations will come in handy.
\begin{obs}
\label{dict}
Let $p=p(n) \rightarrow 0$, $d=pn \rightarrow \infty$, and
$r = \Ln \left(\frac{d}{\ln^2 d} (1+o(1)) \right)$.
Then the following identities hold.
\begin{enumerate}[(i)]
	\item \label{dicti} $r =\frac{\ln d - 2 \ln\ln d +o(1)}{\ln q}= \frac{\ln d }{p} (1+o(1)) = \frac{n\, \ln d }{d} (1+o(1))$.
	\item \label{dictiii} $(1-p)^r = \frac{\ln^2 d}{d}\, (1+o(1)) \rightarrow 0$.
\end{enumerate}
\end{obs}
In the next lemma we establish  that when $r$ is in the range of our interest, then
the expected number of dominating sets of size $r+1$
is much more than the expected number of dominating sets of size $r$.

\begin{lem}
\label{expecJump}
Let $p=p(n) \rightarrow 0$, $d=pn \rightarrow \infty$.
For $\ell < n/2$, $\expec (X_l) < \expec (X_{l+1})$.
Furthermore, for every $r=r(n)=  \Ln \left(\frac{d}{\ln^2 d} \right) +o(1/p)$
and $\alpha = o(1/p)$ we have
\[\frac{\expec (X_{r+\alpha})}{\expec (X_r)} =\exp \Big(  (1+o(1))\, \alpha \, \ln^{2} d\Big).\]
\end{lem}
\begin{remark}
Note that $r =  \Ln \left(\frac{d}{\ln^2 d} \right) +o(1/p)$ if and only if
$r = \Ln \left(\frac{d}{\ln^2 d} (1+o(1)) \right)$. Hence, we can use
Observation~\ref{dict}.
\end{remark}
\begin{proof}
First, let us note that, since $\alpha = o(1/p)$,
\begin{equation}\label{shorty}
(1-p)^{\alpha} = 1- p\alpha (1+o(1)).
\end{equation}
For the lower bound, note that for every $\ell$, $0 \leq \ell < n/2$,
\begin{align*}
\frac{\expec (X_{\ell+1})}{\expec (X_{\ell})}
	& = \frac{{n\choose \ell+1}(1-(1-p)^{\ell+1})^{n-\ell-1}}{{n\choose \ell} (1- (1-p)^\ell)^{n-\ell}}\\
	& = \frac{n-\ell}{\ell+1}  \exp  \left[ (n-\ell) \sum_{k=1}^\infty
		\frac{(1-p)^{\ell k}}{k} \left( 1- (1-p)^k \left(1- \frac{1}{n-	\ell}\right)\right)\right]\\
	& >
	\exp  \left[ (n-\ell) (1-p)^\ell  \left( 1- (1-p)\left(1- \frac{1}{n-\ell}\right)\right)\right]\\
	& >
	\exp  \left[(n-\ell)(1-p)^\ell p  \right].
\end{align*}
In particular, $\expec (X_l) < \expec (X_{l+1})$ for all $\ell < n/2$.
Using the above, we estimate the quotient from the lemma as follows:
\begin{align*}
\frac{\expec (X_{r+\alpha})}{\expec (X_{r})} & = \frac{\expec (X_{r+\alpha})}{\expec (X_{r+\alpha -1})} \cdot \frac{\expec (X_{r+\alpha-1})}{\expec (X_{r+\alpha -2})} \cdot \cdots \cdot \frac{\expec (X_{r+1})}{\expec (X_{r})}\\
& \geq \exp  \left[\sum_{i=0}^{\alpha -1} (n-(r+i)) (1-p)^{r+i} p\right]\\
& \geq \exp  \left[(n-r-\alpha)(1-p)^r \frac{1-(1-p)^\alpha}{p}p\right]\\
& = \exp \left[  (1+o(1)) \alpha\, \ln^{2} d\right],
\end{align*}
where in the last equality we used that
$n- r -\alpha = n(1+o(1))$
by Observation~\ref{dict}$(i)$,
as well as \eqref{shorty} and the formula of Observation~\ref{dict}$(ii)$.

For the upper bound note that for every $\ell\geq 0$, we have
\begin{align*}
\frac{\expec (X_{\ell+1})}{\expec (X_{\ell})}
	& = \frac{{n\choose \ell+1}(1-(1-p)^{\ell+1})^{n-\ell-1}}{{n\choose \ell} (1- (1-p)^{\ell})^{n-\ell}}
	 = \frac{n-\ell}{\ell+1} \left( 1+ \frac{p(1-p)^{\ell}}{1-(1-p)^{\ell}} \right)^{n-\ell} \cdot \frac{1}{1-(1-p)^{\ell+1}}.
\end{align*}
To estimate from above, we use again the telescopic product, and that by Observation~\ref{dict} $(ii)$
the last factor of the above expression is less than $2$ for large $n$ and $\ell = r+o(1/p)$. Thus we obtain
\begin{align*}
\frac{\expec (X_{r+\alpha})}{\expec (X_{r})}
	& \leq \left(\frac{n}{r}\right)^\alpha \cdot
	\exp  \left[ \sum_{i=0}^{\alpha -1} (n-(r+i)) \frac{(1-p)^{r+i} p}{1-(1-p)^{r+i}} \right] \cdot 2^\alpha \\
	& \leq \left(\frac{2n}{r}\right)^\alpha \cdot \exp  \left[ \frac{n(1-p)^r p}{1-(1-p)^{r}} \cdot \sum_{i=0}^{\alpha -1} (1-p)^i \right]\\
	& = \left(\frac{2n}{r}\right)^\alpha  \cdot \exp  \left[ (1+o(1)) n p  (1-p)^r\frac{1-(1-p)^\alpha}{p} \right]\\
	& \leq \exp \Big[ \alpha \left(\ln (2n/r)  +  (1+o(1))n  p (1-p)^r\right) \Big] \\
	& =\exp \lt[  (1+o(1)) \alpha\, \ln^{2} d\rt] ,
\end{align*}
where in the last inequality we again use~\eqref{shorty}, and the last equality holds since by Observation~\ref{dict}~$(i)$
\[ \ln ( 2n/r) = \ln \left( (2+o(1))d/ \ln d\right) =  o\left(\ln^2 d\right).\]
\end{proof}
By the previous lemma, in our range of interest the expectation $\expec(X_r)$ is strictly increasing in $r$ and
grows by a factor of  $\exp \left[ (1+o(1))\ln^2 d  \right]$ with each increase of $r$ by $1$.
Recall that our critical dominating set size is
\[\myr = \min \lt\{\, r\ | \ \expec(X_{r}) \geq \exp \left[- \ln d\right] \rt\} -1 .
\]
\begin{lem}
\label{expecSparse}
Let $p\rightarrow 0$ and $d=np \rightarrow \infty$.
Then $\myr$ is of the form
$\myr = \Ln \left(\frac{d}{\ln^2 d} \, (1+o(1)) \right)$.
Furthermore
\begin{enumerate}[(i)]
\item $\expec (X_{\myr}) \rightarrow 0$, and
\item $\expec (X_{\myr +2}) \geq \exp \lt[(1+o(1))\ln^2 d \rt] \rightarrow \infty$.
\end{enumerate}
\end{lem}
\begin{proof}
First let $r =  \left\lfloor \Ln \left(\frac{d}{\ln^2 d} \right) \right\rfloor$.
Then
Observation~\ref{dict}
applies, so that
\begin{align*}
\expec (X_{r})
	&= \binom{n}{r} (1-(1-p)^r)^{n-r} \\
	&\leq \exp \Big[ r\, \ln \left(\frac{ne}{r}\right) - (n-r) (1-p)^r  \Big]\\
	&\leq
	\exp 	\Bigg[\left(\frac{\ln d}{\ln q} - \frac{2\, \ln \ln d}{\ln q}(1+o(1)) \right)
		\cdot \ln \left( \frac{de}{\ln d}(1+o(1)) \right)\\
	&\qquad \qquad \qquad - \frac{n\ln^2 d}{d} + \frac{\ln^3 d}{dp}(1+o(1))
\Bigg]\\
	&= \exp \Bigg[ \left( \frac{1}{\ln q} - \frac{1}{p}
\right) \, \ln^2 d
	-\frac{3 \, \ln d\, \ln \ln d}{\ln q}(1+o(1)) \Bigg],
\end{align*}
where the second inequality follows from Observation~\ref{dict},
and in the last equality we use the fact that $\frac{\ln^3 d}{dp}=o\lt(\frac{\ln d\ln\ln d}{\ln q}\rt)$.
Now,
\[ \frac{1}{\ln q} - \frac{1}{p}
= \frac{1}{p+p^2/2+\Ocal\lt(p^3\rt)}- \frac{1}{p} = -0.5 + o(1).\]
Therefore,
$\expec(X_{r}) \leq \exp \Big[ (-0.5+o(1)) \, \ln^2 d \Big] < \exp [- \ln d ]$
for large $n$, so
$\myr \geq \left\lfloor \Ln \left(\frac{d}{\ln^2 d}\right)
\right\rfloor$ by definition.

For the upper bound let us redefine
$r =  \left\lceil \Ln \left(\frac{d}{\ln^2 d} \lt(1+p +\frac{1}{\ln \ln d}\rt) \right) \right\rceil$.
Then $r = \Ln \left(\frac{d}{\ln^2 d} \, (1+o(1)) \right)$, so
Observation~\ref{dict}
applies. Thus,
for $n$ being sufficiently large,
\begin{align*}
\expec (X_r) &= \binom{n}{r} \left(1- (1-p)^{r} \right)^{n-r} \\
	&\geq \exp \Big[ r (\ln (n/r ) -(n-r ) (1-p)^{r} -(n-r ) (1-p)^{2r} \Big] \\
	&\geq \exp \Bigg[\left(\frac{\ln d}{\ln q} - \frac{2\, \ln \ln d}{\ln q}(1+o(1)) \right)
		\cdot \ln \left( \frac{d}{\ln d}(1+o(1)) \right)\\
	&\qquad \qquad- \frac{n\ln^2 d}{d \lt(1+p+\frac{1}{\ln \ln d}\rt)} + \frac{\ln^3 d}{dp}(1+o(1))
		- \frac{\ln^4 d}{dp}(1+o(1))\Bigg]\\
	&= \exp \Bigg[ \left( \frac{1}{\ln q} - \frac{1}{p \lt(1+p+\frac{1}{\ln \ln d}\rt)} \right) \, \ln^2 d
		-\frac{3 \ln d \, \ln \ln d}{\ln q}(1+o(1)) \Bigg],
\end{align*}
where in the equality we use the fact that $\frac{\ln^4 d}{dp}=o\lt(\frac{\ln d\ln\ln d}{\ln q}\rt)$.
Now,
\begin{align*}
\ln^2 d\lt( \frac{1}{\ln q} - \frac{1}{p\lt(1+p+\frac{1}{\ln \ln d}\rt)}\rt)
&=\frac{\ln^2 d}{\ln q} \lt( 1- \frac{p+p^2/2 +\Ocal (p^3)}{p\lt(1+p+\frac{1}{\ln \ln d}\rt)}\rt)\\
&= \frac{\ln^2 d}{\ln q}\lt( \frac{p}{2} + \frac{1}{\ln \ln d} +\Ocal (p^2) \rt) (1+o(1))\\
&= \lt(\frac{\ln^2 d}{2} +\frac{\ln^2 d}{\ln q\, \ln \ln d}\rt) (1+o(1)).
\end{align*}
Therefore,
\begin{align*}
\expec(X_{r}) &> \exp \Big[  (0.5+o(1))\ln^2 d \Big] > \exp [- \ln d ]
\end{align*}
for large $n$, so by definition,
$\myr <  \left\lceil \Ln \left(\frac{d}{\ln^2 d} \lt(1+p+\frac{1}{\ln \ln d}\rt) \right) \right\rceil$.

Part $(i)$ then follows from the minimality of $\myr$ and since $d\rightarrow \infty$.

Part $(ii)$ follows from the definition of $\myr$ and by Lemma~\ref{expecJump}:
\[	\expec(X_{\myr +2}) = \exp \lt[ (1+o(1))\ln^2 d\rt ] \cdot \expec (X_{\myr +1}) \geq \exp \lt[ (1+o(1))\ln^2 d -\ln d\rt]\rightarrow\infty .\]
\end{proof}
\noindent
Note that from part $(i)$ of the previous lemma it follows by the standard first moment argument that
\[ \prob (D(G) \leq \myr ) = \prob (X_{\myr}  >0)
				 \leq \expec(X_{\myr})
				 \rightarrow 0.\]
Hence, $\prob(D(G) \geq \myr + 1) \rightarrow 1$.
This proves the lower bound of the interval of concentration in the
sparse range of the edge probability $p$ in Theorem~\ref{MAIN}.

Part $(ii)$ is of course only the first step in deducing $\prob (D(G) \leq \myr +2) \rightarrow 1$ for
the upper bound of the interval. We provide the full analysis in Section~\ref{sec:2ptConc}, using the second moment method.

%
\subsection{The dense case}
Now we take a look at the behaviour of the expected number of dominating sets of size $r$
when the edge probability $p$ tends to $1$ moderately fast.
\begin{lem}
\label{expecDense}
Let $p \rightarrow 1$ such that $\frac{1}{1-p}= q \leq n$.
For $r = \Ln \left(\frac{n\ln q}{\ln^2 n} \right)$ we have
\[\expec (X_{\lfloor r \rfloor}) \rightarrow 0.\]
\end{lem}
\begin{proof}
Note that $r = \frac{\ln n - 2\ln \ln n +\ln \ln q}{\ln q} \leq \frac{\ln n}{\ln q} \leq \ln n$
and $(1-p)^r = \frac{\ln^2 n }{n \ln q} \rightarrow 0$ since $q \rightarrow \infty$.
\begin{align}
\expec (X_{\lfloor r\rfloor})
	&= \binom{n}{\lfloor r\rfloor}
	\cdot \left(1-(1-p)^{\lfloor r\rfloor} \right)^{n-\lfloor r\rfloor} \nonumber \\
	&\leq \left( \frac{ne}{\lfloor r\rfloor} \right)^{\lfloor r\rfloor}
	\cdot \exp \left[ -(n-\lfloor r\rfloor)(1-p)^{\lfloor r\rfloor} \right] \nonumber \\
	&\leq \exp \Big[ r \Big(\ln n + 1 -\ln r \Big) - (n-r)(1-p)^r \Big] \nonumber \\									
&= \exp \Big[ (-2\ln \ln n +\ln \ln q) \frac{\ln n}{\ln q} + r((1-p)^r +1 - \ln r)\Big]
	\rightarrow 0,\nonumber
\end{align}
since $q \leq n$ and $r((1-p)^r +1 - \ln r) \leq r(2-\ln r)$ is bounded from above by the constant $e$.
\end{proof}

One can also show that, analogously to the sparse case, $\expec (X_{\lfloor r \rfloor +2})  \rightarrow \infty$.
Since we do not need this fact further, we omit the calculation.

%
%
\section{The variance}
\label{sec:2ptConc}
In this section, we prove Theorem~\ref{MAIN}.
Since the validity of the theorem  was shown already
for constant $p$ in~\cite{wg2001},
we restrict our attention to the cases when $p$ tends to 0 or 1.
We will refer to three cases
\begin{itemize}
\item {\em the sparse case}, when $p \rightarrow 0$ but $p \gg \frac{\ln ^2 n}{\sqrt{n}}$.
	In this case recall that
	$\myr = \min \{\, r\ | \ \expec(X_{r}) \geq \exp \left[- \ln d\right] \} -1$
	and let us set $r = \myr +2$;
\item {\em the dense case}, when $p \rightarrow 1$ but $q = \frac{1}{1-p} \leq n$.
	In this case we set $\myr = \Ln \left(\frac{n\ln q}{\ln^2 n} \right)$
	and $r= \lfloor \myr \rfloor +2$;
\item {\em the very dense case}, when $q = \frac{1}{1-p} > n$.
\end{itemize}
The third case is straightforward and will be treated separately at the end of this section.

The calculations for the sparse and the dense case are often identical,
so we treat these cases in parallel.
We want to apply Chebyshev's Inequality and conclude that
\[\prob (X_{r} = 0) \leq \frac{\var (X_{r})}{\expec (X_{r})^2} \rightarrow 0.\]
In the proof we will stumble over some expressions more than once,
so we bundle the information of their asymptotic behaviour
in the next observation.

\begin{obs}
\label{dictVar}
In both, the sparse and the dense case, we have,
\begin{enumerate}[$(i)$]
\item $(1-p)^r \rightarrow 0$,
\item $r^2 = o(n)$,
\item $r(1-p)^r \rightarrow 0$,
\item $n (1-p)^{2r-1} \rightarrow 0$.
\end{enumerate}
\end{obs}
\begin{proof}
We treat the two cases separately. \\
{\em The sparse case.}
By Lemma~\ref{expecSparse}, $\myr$ and thus $r=\myr +2$
are of the form $\log_q\left(\frac{d}{\ln^2 d}(1+o(1)) \right)$,
and thus Observation~\ref{dict} applies.
Part $(i)$ is just Observation~\ref{dict} $(ii)$.
For the rest we use Observation~\ref{dict} and that $p\gg \frac{\ln^2 d}{\sqrt{n}}$ and $d\rightarrow \infty$:
\[ \frac{r^2}{n} = \frac{\ln^2 d}{np^2} (1+o(1)) \rightarrow 0, \qquad
	r(1-p)^r = \frac{\ln^3 d}{np^2} (1+o(1)) \rightarrow 0, \quad \text{ and } \quad
	n(1-p)^{2r} = \frac{\ln^4 d}{np^2}(1+o(1)) \rightarrow 0. \]
{\em The dense case.}
By definition, $r
= \left\lfloor\frac{\ln n - 2\ln\ln n +\ln\ln q}{\ln q}\right\rfloor +2$.
Therefore, for large $n$,
\[  \Ln \left(\frac{n\, \ln q}{\ln^2 n}\right) + 1
\leq r \leq \frac{\ln n}{\ln q} +2,\]
since $q\leq n$.
Hence $r\geq 1$ and part $(i)$ follows since $p\rightarrow 1$.
For part $(ii)$ we have
\[\frac{r^2}{n} \leq \frac{1}{n}\left(\frac{\ln n}{\ln q}+2 \right)^2 =o(1).\]
For part $(iii)$ and $(iv)$ we have
\begin{align*}
r(1-p)^r &\leq \left(\frac{\ln n}{\ln q}+2 \right) \frac{\ln^2 n}{n \ln q} (1-p) =o(1), \\
n(1-p)^{2r-1} &\leq \frac{\ln^4 n}{n\ln^2 q} (1-p)= o(1).
\end{align*}
\end{proof}
%
%
%
%
%
For the variance $\var (X_r) = \expec(X_r^2)-\expec(X_r)^2$
we need to calculate $\expec(X_r^2) $.\\
Let $I_A$ be the indicator
random variable of the event that subset $A\subseteq [n]$ is dominating.
Then
\begin{align}
\label{eq:expecSquared}
\expec (X_r^2)
	&= \sum_{A,B \in \binom{[n]}{r}} \expec (I_A \cdot I_B)
	= \sum_{A \in \binom{[n]}{r}} \sum_{s=0}^r
		\sum_{\substack{B\in {[n]\choose r}\\ |A\cap B| =s}}   \!\!\!
		\expec (I_A \cdot I_B)\nonumber \\
\shortintertext{Now, for $A,B \in \binom{[n]}{r}$, $|A\cap B| =s$, we have}
\expec (I_A \cdot I_B)
	&\leq \prob \Big(A \text{ dominates }
		\overline{A \cup B} \, \land \,
		B \text{ dominates } \overline{A \cup B} \Big)  \nonumber \\
  	&= \prob \Big( \forall x\in \overline{A \cup B}
  		\quad \exists \, y_1 \in A \cap \Gamma (x) \,
  		\land \, \exists \, y_2 \in B\cap \Gamma (x) \Big) \nonumber \\
  	&= \Big( 1 - 2(1-p)^r + (1-p)^{2r - s} \Big)^{n-2r+s}. \nonumber \\
\shortintertext{So, }
\var (X_r)
	&\leq \binom{n}{r} \sum_{s=0}^r {r\choose s}{n-r\choose r-s}
		\Big( 1 - 2(1-p)^r + (1-p)^{2r - s} \Big)^{n-2r+s} - \expec (X_r)^2.
\end{align}
First we see that the $s=0$ term of the sum is asymptotically at most $\expec (X_r)^2$:
\begin{align*}
& \binom{n}{r} \binom{n-r}{r} \left( 1 - 2(1-p)^{r} + (1-p)^{2r} \right)^{n-2r} \\
&\qquad = \binom{n}{r}^2 \Big(1-(1-p)^r\Big)^{2(n-r)} \,
	\frac{\binom{n-r}{r}}{\binom{n}{r}}
	\Big(1-(1-p)^r \Big)^{-2r} \\
&\qquad \leq \expec(X_r)^2 \cdot \exp \left[2r (1-p)^r + 2r(1-p)^{2r} \right] \\
&\qquad = \expec(X_r)^2 (1+o(1))
\end{align*}
The inequality holds because $(1-p)^r \rightarrow 0$ by Observation~\ref{dictVar} $(i)$.
The final conclusion follows from Observation~\ref{dictVar} $(iii)$.

Now we estimate the remaining terms of the sum~\eqref{eq:expecSquared}.
It turns out that the term of $s=1$ dominates the rest:
Let
\[f(s) = {r\choose s}{n-r\choose r-s}\Big( 1 - 2(1-p)^r + (1-p)^{2r - s} \Big)^{n-2r+s}.\]
We have just seen that
$\var (X_r) \leq \binom{n}{r} \sum_{s=1}^r f(s) + o (\expec(X_r)^2 )$.
We will prove that for large $n$,
\begin{equation} \label{varFinal}
\binom{n}{r} \sum_{s=1}^r f(s) \leq 3\, \binom{n}{r} f(1) .
\end{equation}
First let us show that indeed,~\eqref{varFinal} implies $\var (X_r) = o( \expec (X_r)^2).$
\begin{align} \label{finally}
\frac{\binom{n}{r} f(1)}{\expec (X_r)^2}
	&= \frac{r \binom{n-r}{r-1} \Big( 1-2(1-p)^r+(1-p)^{2r-1} \Big)^{n-2r+1}}
			{\binom{n}{r} (1-(1-p)^r)^{2(n-r)}} \nonumber \\
	&= \frac{r^2}{n} \cdot \frac{\binom{n-r}{r-1}}{\binom{n-1}{r-1}}
		\cdot \left(1 + \frac{p(1-p)^{2r-1}}{(1-(1-p)^r)^2} \right)^{n-r}
		\cdot \left(\frac{1}{1-2(1-p)^r+(1-p)^{2r-1}} \right)^{r-1} \nonumber \\
	&\leq \frac{r^2}{n} \exp \Bigg[ \Big((n-r)p (1-p)^{2r-1} + (r-1)2(1-p)^r \Big) (1+o(1)) \Bigg] \nonumber\\
	& \rightarrow 0,
\end{align}
where in the last inequality we use again that $(1-p)^r \rightarrow 0$ by Observation~\ref{dictVar} $(i)$.
The final conclusion follows by Observation~\ref{dictVar} $(ii)$, $(iii)$ and $(iv)$.
Thus, $\var (X_r) = o( \expec (X_r)^2 )$ follows.

So we only need to show that~\eqref{varFinal} holds.
We consider the expression
\begin{align*}
\frac{f(1)}{f(s)}
	& = \frac{ r{n-r\choose r-1}\Big( 1 - 2(1-p)^r + (1-p)^{2r - 1} \Big)^{n-2r+1}}
			{{r\choose s}{n-r\choose r-s}\Big( 1 - 2(1-p)^r + (1-p)^{2r - s} \Big)^{n-2r+s}}.
\end{align*}
Note first that for every $2\leq s \leq r$,
\begin{align}
\label{varBinomAllgemein}
\frac{\binom{n-r}{r-1}}{\binom{r}{s} \binom{n-r}{r-s}}
	&= \frac{(n-2r+s)_{s-1}}{(r-1)_{s-1}} \cdot \frac{s!}{(r)_s}
	\geq \left( \frac{n-2r}{r} \right)^{s-1} \cdot \left( \frac{1}{r} \right)^s \nonumber \\
	&= \exp \left[ (s-1) \ln \left( \frac{n-2r}{r} \right) - s \ln r \right].
\end{align}
Also, since $(1-p)^r \rightarrow 0$ by Observation~\ref{dictVar} $(i)$ and $r=o(n)$ by Observation~\ref{dictVar} $(ii)$,
for every $2\leq s \leq r$ we have that
\begin{align}
\label{varProbAllgemein}
\frac{\Big( 1 - 2(1-p)^r + (1-p)^{2r - 1} \Big)^{n-2r+1}}
	{\Big( 1 - 2(1-p)^r + (1-p)^{2r - s} \Big)^{n-2r+s}}
	& \geq \left( 1 - \frac{(1-p)^{2r-s} - (1-p)^{2r-1}}{1-2(1-p)^r +(1-p)^{2r-s}}\right)^{n-2r+1}\nonumber \\
& = \Big( 1- (1-p)^{2r-s}(1-(1-p)^{s-1})(1+o(1)) \Big)^{n-2r+1}\nonumber \\
& = \exp \Big[ - n (1-p)^{2r-s} (1-(1-p)^{s-1}) (1+o(1)) \Big].
\end{align}

\noindent
From now on, we need to separate the sparse and the dense case.
\subsection{The sparse case.}
\label{thesparsecase}
Recall that in this case $p\gg \frac{\ln^2 d}{\sqrt{n}}$ and $d\rightarrow \infty$.
In order to deduce~\eqref{varFinal}, we show for $n$ large enough
that for every $2\leq s \leq \min\{ \ln n, 1/\sqrt{p}\}$,
we have $f(1) \geq \ln n\, f(s)$,
and for every $\min\{\ln n,1/\sqrt{p}\} \leq s \leq r$, we have that $f(1) \geq r f(s)$.
Then we have that
\begin{align*}
\sum_{s=1}^r f(s) & \leq f(1) + \ln n\, \max \big\{ f(s): 2\leq s \leq \min\lt\{ \ln n, 1/\sqrt{p}\rt\}\big\}\\
	&\qquad + r \max \big\{ f(s) : \min\lt\{ \ln n, 1/\sqrt{p}\rt\} \leq s \leq r \big\}\\
& \leq 3 f(1).
\end{align*}

We split the analysis into three cases. \\

\noindent
{\bf Small range.}
First, suppose $2\leq s \leq \min\lt\{\ln n, 1/\sqrt{p}\rt\}$.
Then, since $s \ll 1/p$, we have that
\begin{align} \label{bernoulli}
(1-p)^s \geq 1-ps \qquad \text{and thus} \qquad (1-p)^s = 1+o(1).
\end{align}
So,
\begin{align*}
\frac{f(1)}{\ln n \cdot f(s)}
	& = \frac{r}{\ln n} \cdot
		\frac{ {n-r\choose r-1}}{{r\choose s}{n-r\choose r-s}}  \cdot
		\frac{ \Big( 1 - 2(1-p)^r + (1-p)^{2r - 1} \Big)^{n-2r+1}}{\Big( 1 - 2(1-p)^r + (1-p)^{2r - s} \Big)^{n-2r+s}}\\
	&\overset{\eqref{varBinomAllgemein},\eqref{varProbAllgemein},\eqref{bernoulli}}{\geq}
	\exp \Bigg[ \ln r - \ln \ln n + (s-1) \ln \left( \frac{n-2r}{r} \right) - s \ln r \\
	&\qquad - np(s-1) (1-p)^{2r-s} (1+o(1)) \Bigg] \\
	&\overset{\eqref{bernoulli}}{=}
	\exp \Bigg[ (s-1) \left( \ln \left( \frac{n}{r} \right) + o(1) - \ln r
	- np(1-p)^{2r}(1+o(1)) \right) -\ln \ln n \Bigg]\\
	&\geq \exp \Bigg[ (s-1) \left( (1+o(1))  \ln \left(\frac{n}{r^2\, \ln n} \right) + o(1)\right)\Bigg] \\
	&\geq 1,
\end{align*}
where the second to last inequality follows from Observation~\ref{dictVar} $(iv)$,
whereas the last one follows since $s\geq 2$, and from
$\frac{n}{r^2\, \ln n} \geq \frac{np^2}{\ln ^3 n} (1+o(1)) \rightarrow \infty$
by Observation~\ref{dict} $(i)$ and since
$p \gg \frac{\ln^{3/2}n}{\sqrt{n}}$. \\

\noindent
{\bf Middle range.}
Now, let $\min \lt\{\ln n,\frac{1}{\sqrt{p}}\rt \} \leq s \leq 0.9\, r$.
Since $\min\lt\{\ln n, \frac{1}{\sqrt{p}}\rt\} \rightarrow \infty$, by~\eqref{varBinomAllgemein} and~\eqref{varProbAllgemein} we obtain
\begin{align}
\label{middleRange}
\frac{f(1)}{r \cdot f(s)}
	& = 	\frac{ {n-r\choose r-1}}{{r\choose s}{n-r\choose r-s}}  \cdot
		\frac{ \Big( 1 - 2(1-p)^r + (1-p)^{2r - 1} \Big)^{n-2r+1}}{\Big( 1 - 2(1-p)^r + (1-p)^{2r - s} \Big)^{n-2r+s}} \nonumber \\
	&\geq \exp \Bigg[ s \ln \left( \frac{n}{r^2} \right) (1+o(1))  - n (1-p)^{2r-s} (1+o(1)) \Bigg].
\end{align}
We will show that $s\, \ln \left(\frac{n}{r^2}\right) \gg  n (1-p)^{2r-s}$
and conclude that $f(1) \geq r\, f(s)$.
First note that $\ln \left( \frac{n}{r^2} \right) \rightarrow \infty$
by Observation~\ref{dictVar} $(ii)$.
Also, by Observation~\ref{dict}
$n(1-p)^{2r} = \frac{\ln^4 d}{np^2} (1+o(1))$.
By differentiating one finds that the function $g(s) := s(1-p)^s$
takes its minima at the endpoints of the interval
$\big[ \min\lt \{\ln n,\frac{1}{\sqrt{p}} \rt\}, 0.9r \big]$.
We check that both values $g\lt(  \min \lt\{\ln n,\frac{1}{\sqrt{p}}\rt \} \rt)$ and
$g(0.9\, r)$ have higher order than $\frac{\ln^4 d}{np^2}$, and hence for large enough $n$
the exponent of~\eqref{middleRange} is positive,
completing the proof of the middle range.
Firstly by Observation~\ref{dict},
\[0.9\, r (1-p)^{0.9\,r} = 0.9 \, \frac{\ln d}{p} \left(\frac{\ln ^2 d}{d}\right)^{0.9} (1+o(1))
 \gg \frac{\ln ^4 d}{np^2}.\]
Secondly since $p \rightarrow 0$ and by Observation~\ref{dictVar} $(iv)$,
\[\frac{1}{\sqrt{p}} (1-p)^{\frac{1}{\sqrt{p}}}= \frac{1}{\sqrt{p}} (1-o(1)) \gg 1 \gg n(1-p)^{2r}.\]
We need to bound $g(\ln n)$ only if $p < 1/ \ln ^2 n$,
otherwise $\min\lt \{\ln n,\frac{1}{\sqrt{p}}\rt \} = \frac{1}{\sqrt{p}}$.
But then
\[\ln n (1-p)^{\ln n} = \ln n \exp(-p \ln n (1+o(1))) \geq \ln n \exp(-1/\ln n)(1+o(1)) \gg n(1-p)^{2r}.\]
This completes the proof of the middle range. \\

\noindent
{\bf Large range.}
Finally, let $0.9 r \leq s \leq r$. Then
\begin{align}
\label{large-exp}
\frac{f(1)}{r\, f(s)}
	&= \frac{\binom{n-r}{r-1}}{\binom{r}{s}\binom{n-r}{r-s}}
		\cdot \frac{\Big( 1-2(1-p)^r+(1-p)^{2r-1} \Big)^{n-2r+1}}{\Big( 1-2(1-p)^r+(1-p)^{2r-s} \Big)^{n-2r+s}} \nonumber\\
	&= \expec (X_r) \cdot \frac{\binom{n-r}{r-1}}{\binom{n}{r}\binom{r}{s}\binom{n-r}{r-s}}
		\cdot \left[ \frac{1-2(1-p)^r+(1-p)^{2r-1}}{\Big( 1-2(1-p)^r+(1-p)^{2r-s} \Big)(1-(1-p)^r)} \right]^{n-r} \nonumber\\
	&\qquad	\cdot \frac{\Big( 1-2(1-p)^r+(1-p)^{2r-s}\Big)^{r-s}}{\Big( 1-2(1-p)^r+(1-p)^{2r-1} \Big)^{r-1}}\nonumber\\
	&\geq \expec (X_r) \frac{r (1+o(1))}{n} \cdot
		\frac{1}{\binom{r}{s}\binom{n-r}{r-s}}\cdot
		\left[1+\frac{(1-p)^r(1-(1-p)^{r-s})}{1-2(1-p)^r+(1-p)^{2r-s}}\right]^{n-r}\nonumber\\
	&\geq \exp \Big[(1+o(1))\, \ln^2d -\ln n -2 (r-s)\, \ln n + n(1-p)^r(1-(1-p)^{r-s})(1+o(1)) \Big],
\end{align}
where in the first inequality, we estimated the last factor by $1$ and used that
\[\frac{\binom{n-r}{r-1}}{\binom{n}{r}}  \geq \frac{r}{n}\left( 1-\frac{r-1}{n-r+1}\right)^{r-1}
	= \frac{r}{n} (1+o(1)), \]
since  $r^2 =o(n)$ by Observation~\ref{dictVar}~$(ii)$.
In the last inequality we estimated by
$\binom{r}{s} = \binom{r}{r-s} \leq n^{r-s}$ and
$\binom{n-r}{r-s} \leq n^{r-s}$ as well as used that
$r=o(n)$ and $(1-p)^r =o(1)$, by Observation~\ref{dict} $(i)$ and $(ii)$,
and Lemma~\ref{expecSparse} $(ii)$.
\noindent
Now since $d \gg \sqrt{n}\, \ln^2n$, $(1+o(1))\, \ln^2 d - \ln n \geq 1/4 \ln^2 n (1+o(1))$
in~\eqref{large-exp}.
Note that the $1+o(1)$-function in~\eqref{large-exp} does not depend on $s$.
So for sufficiently large $n$, this expression is larger than $1/2$,
and thus, it is enough to show that for large $n$ and every
$s \in [0.9r, r]$ in the large range
\begin{equation}
\label{largeRangeFinal}
-2(r-s)\, \ln n + \frac{1}{2}n(1-p)^r (1-(1-p)^{r-s}) \geq 0.
\end{equation}

To prove~\eqref{largeRangeFinal}
set $x= r-s$ and consider the function
$h(x) := -2x\, \ln n  + \frac{1}{2}n(1-p)^r(1-(1-p)^x)$
on the interval $[ 0, 0.1r ]$.
By differentiating twice we see that $h(x)$ is concave,
and therefore $h(x) \geq \min \{h(0), h(0.1\,r)\}$
for $0 \leq x \leq 0.1 r$. Now $h(0)=0$.
For the other endpoint, by Observation~\ref{dict} $(i)$ and $(ii)$, and since $d \gg \sqrt{n}$ we have
\begin{align*}
h(0.1 r) &=-0.2 r \ln n + \frac{1}{2}n(1-p)^r(1-(1-p)^{0.1\, r}) \\
	&= -0.2 r \ln n + \frac{r\ln d}{2} (1-o(1)) \rightarrow \infty.
\end{align*}
This finishes the proof of the large range,
and therefore also the proof of Theorem~\ref{MAIN} in the sparse case.
%
%
\subsection{The dense case}
Similarly to the sparse case, we aim to prove $\sum_{s=1}^{r} f(s) \leq 3 f(1)$.
We show in the following that for every $2\leq s\leq r$ we have $f(1) \geq r\, f(s)$.
This then implies~\eqref{varFinal}. \\
Recall that $q=\frac{1}{1-p} \rightarrow \infty$, $q\leq n$ and that
$r = \left\lfloor \log_q \left( \frac{n\, \ln q}{\ln ^2 n} \right) \right\rfloor +2$,
that is
\[ \frac{\ln n +\ln \ln q - 2 \ln \ln n}{\ln q} +1 \leq r \leq  \frac{\ln n +\ln \ln q - 2 \ln \ln n}{\ln q} +2.\]
Let $2\leq s\leq r$.
First note that $(1-p)^{s-1} \rightarrow 0$ since $p \rightarrow 1$.
Then by the inequalities~\eqref{varBinomAllgemein} and~\eqref{varProbAllgemein},
\begin{align*}
\frac{f(1)}{r\, f(s)}
	&\geq \exp \Bigg[ (s-1) \ln \left(\frac{n-2r}{r}\right) -  s \ln r
		-n(1-p)^{2r-s} (1-(1-p)^{s-1})(1+o(1)) \Bigg] \\
	&\geq \exp \Bigg[  (s-1) \bigg( \ln n +o(1)  -2\left(\frac{s}{s-1} \right) \ln r \bigg)
		-n(1-p)^{2r-s} (1+o(1)) \Bigg] \\
	&= \exp \Bigg[  (s-1) \ln n (1 +o(1)) -n(1-p)^{2r-s} (1+o(1)) \Bigg],
\end{align*}
since $\ln r \leq \ln \left(\frac{\ln n}{\ln q}\right) \leq \ln \ln n \ll \ln n$ and $\frac{s}{s-1} \leq 2$.
We will show that $(s-1) \ln n \geq 2 n(1-p)^{2r-s}$ for all $2\leq s \leq r$.
To this end, consider the function $h(x) := x \ln n - 2n(1-p)^{2r-1}(1-p)^{-x}$
on the interval $[ 1, r-1]$.
Differentiating twice shows that $h$ is concave,
and thus, $h(x) \geq \min \{h(1),h(r-1) \}$ for $1\leq x \leq r-1$.
Now, $h(1) = \ln n- 2 n(1-p)^{2r-2} = \ln n -o(1) \gg 1$ by Observation~\ref{dictVar} $(iii)$.
Secondly,
\begin{align*}
h(r-1) &= (r-1)\ln n -2n(1-p)^r \\
	&\geq \frac{\ln n +\ln \ln q - 2 \ln \ln n}{\ln q} \, \ln n - 2 \frac{\ln ^2 n}{q \, \ln q} \\
	&\geq \frac{\ln ^2 n}{\ln q} (1-o(1)) \geq \ln n \gg 1,
\end{align*}
where we used that
$(1-p)^r \leq \frac{\ln ^2 n}{n \ln q} (1-p)$,
$q \rightarrow \infty$ and that $q\leq n$.
We conclude that $f(1) \gg r f(s)$ for all $2\leq s \leq r$.
This finishes the proof of the dense case.

\subsection{The very dense case}
\label{sec:denseConcentration}
Now, we complete the picture when $p$ tends to 1.\\
\noindent
Let $p = p(n) \rightarrow 1$ such that $q > n$.
Then, $\Ln \left(\frac{n \ln q}{\ln ^2 n}\right) <1$
since $\frac{n}{\ln ^2 n} < \frac{q}{\ln q}$ for $n \geq 3$.
Thus $r = \lfloor \myr \rfloor + 2 \leq 2$ and we claim that the domination
number is a.a.s.~at most 2.
In fact, when $p \geq 1-1/n$, then the complement of $G\sim \Gnp$
has an isolated vertex a.a.s. But the very same vertex is
a dominating set of size 1 in $G$.

The proof of Theorem~\ref{MAIN} is complete.

%
%
\section{Concentration for smaller values of $p$}
\label{SEC:furtherStuff}

In this section we prove Proposition~\ref{SomeConcentration} and Theorem~\ref{MAINTAL}.

\begin{proof}[Proof of Proposition~\ref{SomeConcentration}]
Let $\myr$ be given by~\eqref{rDefinition}.
Then by Lemma~\ref{expecSparse} and Observation~\ref{dict},
$\myr = \log_q \left(\frac{d}{\ln^2 d}\, (1+o(1)) \right) = \frac{n}{d}\,\ln d\, (1+o(1))$
and $\expec(X_{\myr}) \rightarrow 0$.
Therefore,
\[ \prob (D(G)\leq  \myr )
	= \prob (X_{ \myr} >0)
	\leq \expec (X_{ \myr})
	\rightarrow 0, \]
and hence, $D(\Gnp )> \myr$ a.a.s., proving the lower bound.

For the upper bound, we set $r= n\frac{\ln d}{d}$ and apply the alteration technique from \cite[Theorem 1.2.2]{as2008}.
We take the set $[r]$ and calculate how many vertices are {\em not} dominated by it.
Adding these vertices to $[r]$ results in a dominating set.
Let $Y$ be the number of vertices not dominated by $[r]$. Then
$Y$ is the sum of $n-r$ independent Bernoulli trials with success probability $(1-p)^r$ each.
By Markov's Inequality, $\prob ( Y \leq \expec (Y) \ln\ln d ) \rightarrow 0$ since  $d\rightarrow \infty$.
So with probability tending to $1$, $\Gnp$ has a dominating set of size
\[r + \expec ( Y)\ln\ln d = r + (n-r) (1-p)^r\ln \ln d \leq n\frac{\ln d}{d} +  n \exp \left( -p n \frac{\ln d}{d}\right)\ln\ln d =
n\frac{\ln d}{d} (1+o(1)).
\]
Therefore, $D(\Gnp) \leq n\frac{\ln d}{d}(1+o(1))$ a.a.s.
\end{proof}
\noindent
We now prove Theorem~\ref{MAINTAL}.

\begin{proof}[Proof of Theorem~\ref{MAINTAL}.]
We use Talagrand's Inequality to show concentration in Theorem~\ref{MAINTAL}.
To that end, let us introduce the necessary terminology.
The following setting can be found in~\cite[Chapter~7.7]{as2008}.\\
Let $\Omega = \prod_{i=1}^N \Omega_i$ be the product space of probability spaces $\Omega_i$,
equipped with the product measure.
We say that a random variable $X:\Omega \rightarrow \mathbb{R}$ is {\em Lipschitz}, if
$|X(x)-X(y)| \leq 1$ whenever $x$ and $y$ differ in at most one coordinate.
Further, for a function $f:\mathbb{N} \rightarrow \mathbb{N}$, we say that $X$ is {\em $f$-certifiable}
if whenever $X(x)\geq s$ there exists $I\subseteq [N]$ with $|I|\leq f(s)$ such that for all $y \in \Omega$
with $x_I = y_I$ we have $X(y)\geq s$. We use the following version of Talagrand's Inequality.
\begin{thm}[Talagrand's Inequality]
\label{talagrand}
Let $f:\mathbb{N} \rightarrow \mathbb{N}$ be a function, and suppose $X:\Omega \rightarrow \real$
is a random variable that is Lipschitz and $f$-certifiable. Then for all $a$, $u \in \real$:
\[ \prob \Big(X\leq a-u \sqrt{f(a)} \Big)\cdot \prob(X\geq a) \leq e^{-u^2/4}.\]
\end{thm}
\begin{cor}
\label{lemT1}
For all $b, t \in \mathbb{R}$,
\begin{equation} \label{eq:T1}
	\prob (D(G) \leq b)\cdot \prob (D(G) \geq b+t) \leq e^{-t^2/4(n-b)}.
\end{equation}
\end{cor}
\begin{proof}
We check that Theorem~\ref{talagrand} can be applied to our situation.
For that reason, we identify a graph $G$ with its edge set, and view \Gnp\ as the product of $N=\binom{n}{2}$
Bernoulli experiments with parameter $p$. Let us consider the random variable
$X : \Gnp \rightarrow \real$ defined by $X(G) = n - D(G)$. Clearly, $X$ is Lipschitz, since adding
or deleting an edge changes the domination number (and hence $X$) by at most one.
Further, $X$ is $f$-certifiable, where $f(s)=s$.
To see this, assume $X(G) \geq s$, i.e. $D(G) \leq n-s$.
Then there exists a dominating set $S$ of size $n-s$.
We can choose $s$ edges, one from each $v \in \left( V(G) \setminus S \right)$ to $S$,
which certify that $D(G) \leq n-s$ (more precisely that $S$ is a dominating set).
Clearly, any graph $H$ that contains those $s$ edges will have $D(H) \leq n-s$, or $X(H) \geq s$, respectively.	
Now, it follows by Theorem~\ref{talagrand}, that for all $a,u \in \real$,
\begin{align*} \label{eq:T2}
	&\ &\prob \Big(n - D(G) \leq a - u\sqrt{a}\Big)\cdot \prob \Big(n-D(G) \geq a \Big) &\leq e^{-u^2/4}.
\end{align*}
Substituting $b=n-a$ and $t = u \sqrt{a}$ proves the claim.
\end{proof}
\noindent
To turn Corollary~\ref{lemT1} into a meaningful result let $t=t(n)$ be any function such that $t = \omega (\sqrt{n})$.
If we now set $b$ to be the median in Corollary~\ref{lemT1}, then
we obtain $\prob (D(G) \geq m+t) \leq 2 e^{-t^2/4n}$.
Analogously, setting $b+t = m$ gives $\prob (D(G) \leq m-t) \leq 2 e^{-t^2/4n}$.
Hence, Theorem~\ref{MAINTAL} follows.
\end{proof}
%
%
%
%
%
\section{Non-concentration}
\label{SEC:nonConc}
In this section we prove Theorem~\ref{newNonConc}, 
giving a justification for the existence of some lower bound on $p$ in Theorem~\ref{MAIN}.

\begin{proof}[Proof of Theorem~\ref{newNonConc} (a) and (c)]
Let us first describe the common idea of the two proofs.
We assume to the contrary that for $G\sim \Gnp$, $D(G)$ is in some interval 
$I$ of integers a.a.s.~(in case of (a) $I$ is an interval of length $Kn\sqrt{p}$,  
in case of (c) $I=[1, \myr + c\frac{\myr}{n\sqrt{p}}]$).
Then we delete edges of $G$ with a tiny probability, a probability so
small, that the resulting graph $F \sim \mathcal{G}(n,p')$ (where $p'$
is very close to $p$) shows very similar
properties to $G$ a.a.s. In particular, it will be true that $D(F)\leq I$ still holds
a.a.s.
On the other hand we will also show that with positive probability
the deletion process ruins every single dominating set with size from $I$,
a contradiction.

To conclude the similarity of the graph $G$ and the graph $F$ after the deletion
we need the following proposition involving convex graph properties.
A graph property (set of graphs) $Q$ is called {\em convex},
if for any three graphs $G\subseteq F\subseteq H$, from $G\in Q$ and $H\in Q$ 
one obtains $F\in Q$.
Since the graph property defined by the domination number being in some interval $I$ is
convex, the following proposition can be applied.

\begin{prop}
\label{similarity}
Let $Q$ be a convex graph property, $p(1-p)\binom{n}{2} \rightarrow \infty$, $x\in \mathbb{R}$, and set $p' = p + x\frac{\sqrt{p}}{n}$.
Further, suppose that  $G\sim \Gnp$ and $F \sim \mathcal{G}(n,p')$. Then
\[ G \in Q \text{ a.a.s. } \Rightarrow F \in Q \text{ a.a.s. }\]
\end{prop}

\begin{proof}
This follows easily from Theorem~2.2 (ii) in~\cite{bollobasbook}:
if $Q$ is a convex graph property and $p(1-p)\binom{n}{2} \rightarrow \infty$, then {\Gnp}
has property $Q$ a.a.s. if and only if for all fixed $X$ the graph $\mathcal{G}(n,M)$ has $Q$ a.a.s.,
where $M= \left\lfloor p\binom{n}{2} + X\sqrt{p(1-p)\binom{n}{2}}\right\rfloor$.
\end{proof}

Let us start with the proof of (a) and assume that for some $c> 0$ and
$K > 0$ there is a probability 
$p(n)=p \leq \frac{c}{n}$ and interval $I=[i_1,i_2]\subseteq [n]$ of length 
$|I|< Kn\sqrt{p}$, such that $D(G)\in I$ a.a.s. 
We apply the two-round procedure as described above: we first 
draw $G \sim \Gnp$ then we delete every edge of $G$ with probability 
$p'':= 4K e^{2c} \lt(n\sqrt{p}\rt)^{-1}$, these choices being independent.
In the new graph $F$ every edge occurs with probability
$p' = p\lt(1-p''\rt) = p - 4K e^{2c} \frac{\sqrt{p}}{n}$, hence $F\sim \mathcal{G}(n,p')$.
By Proposition~\ref{similarity}, we know that
\begin{equation}
\label{DomF}
	D(G) \in I \text{ a.a.s. } \Rightarrow D(F) \in I \text{ a.a.s. }
\end{equation}
In the following we will show that $D(F)\not \in I$ a.a.s. This contradiction completes the proof. 

By a standard application of the second moment method  
the number of isolated edges in $G\sim\Gnp$ is concentrated a.a.s.~around its 
mean for $p\rightarrow 0$. 
That is, the number of isolated edges is $\frac{1}{2}n^2pe^{-2np}(1+o(1))$ a.a.s. 
At least $\frac{1}{2}(p''\frac{1}{2}n^2pe^{-2np}) \geq K n \sqrt{p}$ of these edges 
are deleted in the second round of our procedure a.a.s. The deletion of any isolated edge
increases the domination number by one, thus the domination number of $F$ 
is a.a.s. $Kn\sqrt{p}$ larger than the domination number of $G$. 
Hence a.a.s. the domination number of $F$ is in the interval $[i_1+ Kn\sqrt{p}, n]$,
which is disjoint from $I$, since the length of $I$ is less than $Kn\sqrt{p}$.
Hence the domination number of $F$ is not in $I$ a.a.s., which provides  
the contradiction.

Let us now turn to the proof of part (c) and suppose for the sake of contradiction that
there exist a constant $c> 0$ and probability $p(n)=p$ with $1/n\ll p \ll 1$
such that the domination number of $\Gnp$ is contained in $I =[1, r]$ a.a.s., 
where $r = \myr+C$ with $C=\left\lfloor c\frac{\myr}{n\sqrt{p}}\right\rfloor$.
Recall that according to~\eqref{rDefinition} and Lemma~\ref{expecSparse}, for $p\gg 1/n$ we have
$$\myr = \Ln \left(\frac{d}{\ln^2 d} \, (1+o(1)) \right).$$
Note that $C = o(1/p)$, 
so Observation~\ref{dict} and Lemma~\ref{expecJump} apply for $r$.

We apply a similar two-stage random procedure. First, we draw a graph $G$ from \Gnp\ 
and then we delete every edge with probability $p'' := 4c(\sqrt{p}n)^{-1}$, again
these choices being independent. 
As above, in the new graph $F$ every edge occurs with probability
$p' = p\lt(1-p''\rt) = p - 4c\frac{\sqrt{p}}{n}$, hence $F\sim \mathcal{G}(n,p')$.
By Proposition~\ref{similarity}, we know that 
$D(G) \in I \text{ a.a.s. } \Rightarrow D(F) \in I \text{ a.a.s.}$ 
In the following we will show that $D(F)\not \in I$ a.a.s. This contradiction completes the proof.

\newcommand{\cru}{\ensuremath{|\mathcal{C}_G(S)|}}
For a subset $S \subseteq V$ of the vertices, we call an edge $e=xs \in E$
{\em crucial with respect to $S$} if $s \in S$, $x \in V \setminus S$ and for all $s' \in S \setminus \{ s \}$, $xs' \notin E$.
That is, a crucial edge is the only connection of $x$ into $S$ in $G$.
In particular, if $S$ was dominating, the deletion of $e$ would result in $S$ not being dominating anymore. Set
\[ \mathcal{C}_G(S)= \{x \in V \setminus S |\, xs \in E(G) \text{ is crucial with respect to } S \text{ for some } s\in S\}. \]
Note that by the definition of a crucial edge, \cru\ counts exactly the number of crucial edges.

We denote by $B$ the event
that there exists a dominating set $S$ of size $r$ in $G$, such that none of its crucial edges have been destroyed, i.e.,
$\mathcal{C}_G(S)\subseteq N_F(S)$.
Clearly, $\bar{B}$ implies $D(F)\not \in I$, since no dominating set
of size $r$ in $G$ remains dominating in $F$ (and hence there is no
smaller dominating set either).
For a subset $S \in \binom{V}{r}$ of the vertices,
let $Y_S$ be the random variable counting the deleted crucial edges w.r.t. $S$
and denote by $D_S$ the event that $S$ is dominating in $G$.
By the union bound we have for every $f=f(n)>0$ that
\begin{align}
\label{FinalCalc1}
\prob (B) &= \prob\left(\exists S\in \binom{[n]}{r}:~D_S \text{ holds and } Y_S =0\right)\nonumber\\
&\leq \sum_{S\in \binom{[n]}{r}}
	\prob\left( D_S \text{ holds,}~Y_S =0 \text{ and }~\cru \geq f\right) \nonumber\\
&\qquad +\sum_{S\in \binom{[n]}{r}} \prob\left( D_S \text{ holds, }~Y_S=0 \text{ and }~\cru < f \right) \nonumber\\
&\leq \sum_{S\in \binom{[n]}{r}} \prob\left(D_S\right)\cdot
	\prob\left(Y_S=0 \ |\   D_S\mbox{ holds and } \cru \geq f\right)\nonumber\\
&\qquad +\sum_{S\in \binom{[n]}{r}} \prob\left(D_S\right)\cdot  \prob\left(\cru < f \ |\
D_S\right).
\end{align}
We start by estimating the second sum.
Let $S \in \binom{[n]}{r}$. We will
observe that $|\mathcal{C}_G(S)|$, conditioned on $S$ being dominating in $G$,
is a binomially distributed random variable.
To this end, for vertices $v \in V\setminus S$, define the events
\begin{align*}
A_v &= \{ v \in \mathcal{C}_G(S)\, |\, D_S \}
= \{ b_v=1\, |\, b_w\neq 0\  \forall\ w \in V\setminus S \},
\end{align*}
where $b_v$ is the number of edges of $G$ among the pairs $\{vs : s \in S\}$.
The edge sets $\{vs : s \in S\}$ are pairwise disjoint, hence the random variables $b_v$, and in turn the
events $A_v$ are mutually independent.

The random variable $\cru$ conditioned on $D_S$ is then the sum of
$(n-r)$ mutually independent random variables: the indicator random variables of the events $A_v$.
These are $1$ with probability
$p^*:= \prob (A_v ) =  \prob(b_v =1) / \prob(b_v \neq 0)  \geq \prob(b_v =1) = rp(1-p)^{r-1}.$
Hence for the expectation of  \cru\ we have
\begin{align*} \mu &:= \expec (\cru \, |\, D_S) = (n-r)p^* \geq (n-r) pr(1-p)^{r-1}.	
\end{align*}
Thus by Chernoff's inequality (see e.g.~\cite{as2008}), plugging in $f=\mu/2$
we have that
\begin{align} \label{secondSummand}
\prob\left(\cru < \frac{\mu}{2}\ \Big| \ D_S\right)
	 &< \exp\left[ -\frac{\mu}{8} \right]
\end{align}
for large enough $n$.

Now, we bound the first sum in \eqref{FinalCalc1}.
Observe that conditioning on $\cru$ being a fixed integer $\ell$, we have $Y_S \sim \Bin(\ell, p'')$,
where $p''$ is the probability that an edge is deleted from $G$.
Furthermore, observe that once we condition on $\cru$ taking one fixed value,
no other information about $G$ influences the distribution of $Y_S$,
especially not the fact that $S$ is dominating in $G$, so
\begin{align*}
\prob\lt(Y_S =0\, |\, \cru = \ell \text{ and } D_S\rt) & =
\prob\lt(Y_S =0\, |\, \cru = \ell \rt)= (1-p'')^\ell.
\end{align*}
Hence, if $G$ was drawn such that $\cru\geq  \mu/2$,
then
\begin{equation}
\label{easy}
\prob\lt(Y_S = 0 | \text{$D_S$ holds and }~\cru \geq  \mu/2\rt)\leq(1-p'')^{\mu/2}.
\end{equation}
Combining \eqref{secondSummand} and \eqref{easy} we obtain in \eqref{FinalCalc1} that
\begin{align}
\prob(B) &\leq \sum_{S\in \binom{[n]}{r}}
	\prob\left(D_S\rt)\cdot \left((1-p'')^{\mu/2} + \exp \lt[- \mu/8 \rt]\right)\nonumber\\
&\leq E(X_r)\cdot\exp \lt[- p''\mu/3 \rt]\label{firstSummand},
\end{align}
since $p''\rightarrow 0$.
Since $p\gg 1/n$, we see that
\[\mu \geq (n-r) pr(1-p)^{r-1}=(1-o(1))r\ln^2 d\]
by Observation \ref{dict}~\eqref{dictiii}.
Hence, plugging it into~\eqref{firstSummand} and using Lemma~\eqref{expecJump}, we obtain
\[\prob(B)\leq \exp\lt(C\ln^2 d (1+o(1))-(1+o(1))x\frac{r \ln^2 d}{3n\sqrt{p}}\rt)\rightarrow 0\]
for $x\geq 4c$.
Hence, $\prob( D(F)\not\in I) \geq \prob\lt(\bar{B}\rt) \rightarrow 1$, 
 a contradiction. This completes the proof of part (c).

\end{proof}

\begin{proof}[Proof of Theorem \ref{newNonConc} $(b)$]
Let $G\sim\Gnp$ and assume that there is some $c>0$ such that for all $\eps >0$ there exists 
a sequence $p=p(n)$ with $\frac{c}{n}\leq p \ll 1$ such that $\Pr(D(G) >\myr +\eps n e^{-2np})\nrightarrow 1$. 
Let $r(n)= \myr + \eps n e^{-2np}$, where we fix $\eps$ later. 
Then there is a $\delta >0$ and a subsequence $\pi_n$ such that 
\begin{equation}\label{aux749}
\Pr\bigg(D\big(G(\pi_n,p(\pi_n)\big) \leq r(\pi_n)\bigg)>\delta.
\end{equation}
We distinguish two cases. 
Either $(1)$ we have that $1/\pi_n\ll p(\pi_n)$, 
or $(2)$ there is yet another subsequence $\tau_n$ of $\pi_n$ and a $K>0$ 
such that $p(\tau_n)\leq K/\tau_n$. 
We will deal with both cases simultaneously, splitting the proof into a short case distinction 
whenever necessary and reaching a contradiction at the end. 
For simplicity of notation, we assume that $\tau$ (and $\pi$) is the identity function, 
as the proof obviously follows the same lines whenever we restrict to a subsequence of
the natural numbers. 

We note first that for any graph $F$ on $n$ vertices that has (at least) $m$ isolated edges that 
\begin{itemize}
\item[$(i)$] $D(F)\leq n-m$, and 
\item[$(ii)$] for any integer $k$ such that $D(F) \leq k \leq n - m$, 
	the number of dominating sets of size $k$ is at least $2^m$. 
\end{itemize}
To see this, let $W\subseteq V$ be the vertex set of $m$ isolated edges. 
Then $V\setminus W$ together with one vertex from each of those $m$ isolated edges 
forms a dominating set of $G$, showing $(i)$. 
For $(ii)$, let $S\subseteq V\setminus W$ be a dominating set of $F[V\setminus W]$ 
of size $k-m$, which exists by the conditions on $k$. 
Then $S$ can be extended to a dominating set of $F$ of size $k$ by taking exactly one 
vertex from each isolated edge, and there are $2^m$ ways of doing so.

As it was noted in the proof of (a) above
the number of isolated edges is $\frac{1}{2}n^2pe^{-2np}(1+o(1))$ a.a.s.~whenever $p\rightarrow 0$. 
We claim that for $n$ large enough 
\begin{equation}\label{aux399}
r(n)=\myr + \eps ne^{-2np} \leq n - \frac{1}{4}n^2pe^{-2np}.
\end{equation}
In case~$(1)$, when $p\gg 1/n$, $n^2pe^{-2np}=o(n)$, $ne^{-2np}=o(n)$ and 
hence $r(n) = o(n)$ by Observation \ref{dict}. 
In case~$(2)$, when $p=\Theta(1/n)$, since $G$ has at least 
$m_0=0.4n^2pe^{-2np}$ isolated edges a.a.s.,
the domination number satisfies $D(G)\leq n-m_0$ 
by $(i)$ a.a.s. Then by $(ii)$, the expected number of dominating sets of size 
$n-m_0$ is at least $(1+o(1))2^{0.4 n^2 p e^{-2np}} \geq 1$ for large enough $n$.
Hence, 
$\myr < n-m_0$ by the definition of $\myr$.
Now, since $ne^{-2np} = \Theta (n^2pe^{-2np})$, we can choose $\eps>0$ 
small enough 
and sandwich $r(n)$ between $\myr$ and $n - 0.25 n^2pe^{-2np}$ 
so that \eqref{aux399} holds. 

By our assumption \eqref{aux749} we have that $D(G)\leq r(n)$ with probability at least 
$\delta$. 
Furthermore, $G$ has at least $\frac{1}{4}n^2pe^{-2np}$ isolated edges a.a.s. 
It follows thus by $(ii)$ and \eqref{aux399} that 
the number $X_{r(n)}$ of dominating sets of size 
$r(n)$ satisfies $X_{r(n)} \geq 2^{n^2p\exp(-2np)/4}$ with probability   
at least $\delta +o(1)$. 
Therefore, $\expec(X_{r(n)}) \geq \delta 2^{n^2p\exp(-2np)/4}(1+o(1))$. 

We would thus obtain a contradiction if we also deduce that 
$\expec(X_{r(n)}) \ll 2^{n^2p\exp(-2np)/4}$ 
in both cases. 
In case $(1)$, the expected number of dominating sets of size 
$\hat{r}+ n e^{-2np}$ is bounded from above by 
$\exp\lt((1+o(1)\ln^2d n\exp(-2np)\rt)=o\lt(2^{n^2p\exp(-2np)/4}\rt)$, 
by Lemma~\ref{expecJump}, yielding the desired contradiction. 

In case $(2)$, we assume that $\frac{c}{n}\leq p\leq \frac{K}{n}$. 
Note that then $2^{n^2p\exp(-2np)/4} \geq 2^{\frac{c}{4}e^{-2K}n} = 2^{c'n}$. 
Let now $k\leq n$ be a general integer again and parametrize $k=c_kn$. 
Then the expected number of dominating sets of size $k$ 
is 
\begin{align}
\expec(X_k) &= \binom{n}{k}\left( 1- (1-p)^k\right)^{n-k}\nonumber\\
		&\leq 2^{H(c_k) n} \left(1-e^{-K\cdot c_k (1+o(1))}\right)^{(1-c_k)n}\nonumber\\
		&= 2^{\left(H(c_k) +(1-c_k)\log_2(1-e^{-K\cdot c_k})\right) (1+o(1)) n},
\end{align}
where $H(x)$ is the binary entropy function. 
Consider the function 
$$f(x):= H(x) +(1-x)\log_2(1-e^{-K\cdot x})$$ on the interval $[0,1]$. 
Using standard calculus one can see that $f(x)$ 
can have at most one local maximum in $[0,1]$. 

Since $f(x)\rightarrow - \infty$ for $x\rightarrow 0$  
and for $x$ being sufficiently close to 1, 
\begin{align*}
f(x) 	
	&= -x\log_2(x)+(1-x)\log_2\left(\frac{1-e^{-Kx}}{1-x}\right) >0,
\end{align*}
$f(x)$ does have a local maximum in $(0,1)$ which is positive.
Now it follows by continuity that $f(c_0)=0$ for a unique $c_0\in (0,1)$ and that 
$f(x)$ is increasing in an $\eps'$-neighbourhood of $c_0$.
Note that $\myr \leq (c_0+o(1)) n$. 
So it is possible to choose $\eps >0$ such that 
for $r(n)= \myr + \eps n e^{-2np} = (c_0+\eps') n$ it holds that 
$0< f(c_0+\eps') < c'$. Hence, 
\begin{align*}
\expec (X_{r(n)}) &\leq 2^{(1+o(1)) f(c_0+\eps') n} \ll 2^{c'n} \leq 2^{n^2p\exp(-2np)/4}. \qedhere
\end{align*}
\end{proof}

%
%

\section{Concluding remarks and open problems}
\label{conclusion}

As was noted in the introduction, part (c) of Theorem~\ref{newNonConc} implies that for $p\leq (\ln n / n)^{2/3}$
the domination number $D(\Gnp)$ is not concentrated a.a.s.~on any constant length interval around $\myr$.
It would be interesting to improve this result in a couple of directions.
On the one hand, it seems reasonable to believe that
the power $-\frac{2}{3}$ in the upper bound  on the edge probability $p$ could be pushed up to $-\frac{1}{2}$, hence
making Theorem~\ref{MAIN} tight up to a polylogarithmic factor.
On the other hand it is unsatisfactory that our current proof for parts (b) and (c) of Theorem~\ref{newNonConc} requires
the extra assumption that the concentration interval is around
$\myr$. 
In fact, the lesson we take out of part (b) of Theorem~\ref{newNonConc} 
is that $\myr$ is by far the false location for any interval of concentration for very small values of $p$.
It would be desirable to obtain a result stating the non-concentration of $D(\Gnp)$ on 
{\em any} constant-length interval, independent of its location --- like we have it for $p=O(1/n)$.

It would also be interesting  to learn more about the concentration of the domination number
of $\Gnp$ in case
$p = \Ocal\lt(\ln^2 n /\sqrt{n}\rt)$.
Theorem~\ref{MAINTAL} does provide a concentration interval of length slightly above $\sqrt{n}$ for all $p$.
Calculating the $o(1)$-term in
Proposition~\ref{SomeConcentration} gives a bound of the order $n\frac{\ln\ln d}{d}$ on the length of the
concentration interval. 
This is better than the one from Theorem~\ref{MAINTAL} for $p \geq \frac{\ln\ln n}{\sqrt{n}}$.

Recall from the introduction that for the range $p=\Theta( \frac{1}{n})$, part (a) of
Theorem~\ref{newNonConc} and Theorem~\ref{MAINTAL} implies that there is 
no concentration on any interval of size $O(\sqrt{n})$, but there is concentration on some
interval of size $\sqrt{n}f(n)$ for any function $f(n)\rightarrow \infty$. Theorem~\ref{MAIN}
resolves the question of the shortest interval length the domination number of $\Gnp$ 
is concentrated a.a.s.~for $p\gg \frac{\log^2 n}{\sqrt{n}}$. 
Ideally one would like to know the length of the shortest concentration interval 
for every $p$. We would conjecture that mostly non-concentration statements are
missing in the range $p \ll \frac{1}{\sqrt{n}}$; Theorem~\ref{newNonConc} is only
a first step in this direction.
Similar questions are also wide open for the chromatic number~\cite{ak1997}: 
we know, for example, that the chromatic number of ${\mathcal G} (n,1/2)$ is 
concentrated a.a.s.~on an interval of length $\sqrt{n}$, 
but we do not know whether it is concentrated on an interval of length two.\\

\textbf{Acknowledgement.} We are grateful to Michael Krivelevich for fruitful discussions
and especially for his suggestions to the proof of part (c) of Theorem~\ref{newNonConc}. 
We also want to thank the anonymous referee for many helpful remarks and 
in particular the proof idea of part (b) of Theorem~\ref{newNonConc}.

\bibliographystyle{abbrv}
\bibliography{referencesAll}

\end{document}